\documentclass[11pt,reqno]{amsart}
\usepackage{amsmath}
\usepackage{amsthm}
\usepackage{amssymb}
\usepackage{enumerate}
\usepackage{amscd}
\usepackage{color}
\usepackage{pb-diagram}
\usepackage{graphicx}
\usepackage[all, cmtip]{xy}
\usepackage{soul}
\usepackage{esint}
\usepackage{hyperref}
\hypersetup{pdftex,colorlinks=true,allcolors=blue}
\usepackage{hypcap}
\usepackage[titletoc]{appendix}

\theoremstyle{plain}

\newtheorem{lemma}{Lemma}
\newtheorem{prop}[lemma]{Proposition}
\newtheorem{coro}[lemma]{Corollary}
\newtheorem{theorem}[lemma]{Theorem}
\newtheorem{conj}[lemma]{Conjecture}

\newtheorem*{prop*}{Proposition}
\theoremstyle{definition}

\theoremstyle{remark}

\numberwithin{equation}{section}
\newenvironment{enumeratei}{\begin{enumerate}[\upshape (i)]}{\end{enumerate}}

\newcommand{\Ric}{\textup{Ric}}

\newcommand{\p}{\partial}

\begin{document}
\title[Holomorphic curvature and canonical bundle]{Negative Holomorphic curvature and positive canonical bundle}

\author{Damin Wu}
\address{Department of Mathematics\\
University of Connecticut\\
196 Auditorium Road,
Storrs, CT 06269-3009, USA}

\email{damin.wu@uconn.edu}

\author{Shing--Tung Yau}
\address{Department of Mathematics \\
				Harvard University \\
				One Oxford Street, Cambridge MA 02138}
\email{yau@math.harvard.edu}
\thanks{The first author was partially supported by the NSF grant DMS-1308837. The second author was partially supported by the NSF grant DMS-0804454.}
\maketitle

\begin{abstract} 
In this note we show that if a projective manifold admits a K\"ahler metric with negative holomorphic sectional curvature then the canonical bundle of the manifold is ample. This confirms a conjecture of the second author.
\end{abstract}

\maketitle

\section{Introduction}

A fundamental question regarding the geometry of a \emph{projective manifold}, i.e., a nonsingular complex projective variety, is to charaterize the positivity of its canonical bundle. 
In algebraic geometry the abundance conjecture predicts that the canonical bundle is semiample if it is nef~\cite{Kaw}. In hyperbolic geometry a conjecture of Kobayashi asserts that the canonical bundle is ample if the manifold is hyperbolic \cite[p. 370]{SKob}, while Lang conjectured that a projective manifold is hyperbolic if and only if every submanifold has big canonical bundle~\cite[p. 190]{Lang}.

From the viewpoint of differential geometry, the canonical bundle can be represented by the Ricci curvature up to a sign. The hyperbolicity is assured by the negativity of holomorphic sectional curvature. The conjectures in hyperbolic geometry naturally connect to the very basic question in complex differential geometry, that is, to understand the mysterious relation between the Ricci curvature and the holomorphic sectional curvature (cf. \cite[p. 181]{Zheng}).

A conjecture of the second author asserts that the holomorphic sectional curvature determines the Ricci curvature, in the sense that if a projective manifold admits a K\"ahler metric of negative holomorphic sectional curvature then it also admits a K\"ahler metric of negative Ricci curvature. 
In view of his resolution of Calabi's conjecture, the second author's conjecture can be stated as below (cf. \cite[Conjecture 1.2]{HLW2})
\begin{conj} \label{co:Yau}
If a projective manifold admits a K\"ahler metric with negative holomorphic sectional curvature then the canonical bundle of the manifold is ample.
\end{conj}

In the present paper, we prove the full statement of Conjecture~\ref{co:Yau}.

\begin{theorem}\label{th:WY}
Let $X$ be an $n$-dimensional projective manifold. If $X$ admits a K\"ahler metric whose holomorphic sectional curvature is negative everywhere, then the canonical bundle $K_X$ is ample.
\end{theorem}

If the aforementioned Kobayashi conjecture \cite[p. 370]{SKob} were known to be true, then Theorem~\ref{th:WY} would follow immediately. On the other hand, if the Kobayashi conjecture fails, then there will be a hyperbolic projective manifold which does not admit any K\"ahler metric with negative holomorphic sectional curvature. This will give a negative answer to the question asked by Kobayashi and Greene-Wu~\cite[p. 84]{GW} in the K\"ahler setting.

Using the hyperbolicity, Theorem~\ref{th:WY} is proved by B. Wong~\cite{Wong81} (see also Campana~\cite{Ca}) for $n = 2$, and proved by Peternell~\cite{P91} for $n = 3$ except possibly for the Calabi--Yau threefolds without rational curves. 
In our early work~\cite{WWY} joint with P. M. Wong, we establish Theorem~\ref{th:WY} under an additional assumption that the Picard number of the projective manifold is equal to 1. Our method is essentially the refined Schwarz type lemma.

Recently, some important progress is made by Heier-Lu-Wong~\cite{HLW1, HLW2}, which combines the algebraic-geometric method and the analytic method. Along with other interesting results, they 
prove Theorem~\ref{th:WY} by assuming the validity of the abundance conjecture, which is known to hold for $n \le 3$. Therefore, they have in particular proved Theorem~\ref{th:WY} for $n = 3$. 

Here is the outline of the proof of Theorem~\ref{th:WY}. Inspired by Heier-Lu-Wong~\cite{HLW2}, we observe that it suffices to prove that  $K_X$ is big. More precisely, the proof can be reduced to show an integral inequality (see \eqref{eq:topKx}) via the Cone Theorem and the Basepoint-free Theorem in algebraic geometry. To get the inequality, a key step is to induce a complex Monge-Amper\`e equation to construct a family of K\"ahler metrics whose Ricci curvature has a uniform lower bound. Then the desired inequality follows from extending the refined Schwarz lemma in our previous papers \cite{WYZ} with F. Zheng and \cite{WWY} with P. M. Wong.

We also provide an alternative proof that bypasses the Basepoint-free Theorem. Instead, we apply the estimates of Monge-Amper\`e equation to show that the family of K\"ahler metrics converges to a K\"ahler-Einstein metric of negative scalar curvature. Then the ampleness of $K_X$ follows from the classical results of Chern and Kodaira.

The following corollary is an immediate consequence of Theorem~\ref{th:WY} and the resolution of the Calabi conjecture (cf. \cite[p. 383]{yau78-calabi}).
\begin{coro} \label{co:KE}
If a projective manifold $X$ admits a K\"ahler metric whose holomorphic sectional curvature is negative everywhere, then $X$ has a K\"ahler-Einstein metric of negative Ricci curvature.
\end{coro}
It is well known that the converse of Theorem~\ref{th:WY} and Corollary~\ref{co:KE} does not hold. For example, the Fermat hypersurface $F_d$ in $\mathbb{CP}^{n+1}$ of degree $d \ge n+3$ has ample canonical bundle, and hence, $F_d$ has a K\"ahler-Einstein metric of negative Ricci curvature. However, $F_d$ does not admit any K\"ahler metric with negative holomorphic sectional curvature, since $F_d$ contains complex lines.

The next corollary follows from Theorem~\ref{th:WY} together with the decreasing property of holomorphic sectional curvature on submanifolds (see for example \cite[Lemma 1]{HWu}).
\begin{coro}\label{co:submfd}
If a nonsingular projective variety $X$ has a K\"ahler metric whose holomorphic sectional curvature is negative everywhere, then the canonical bundle of every nonsingular subvariety of $X$ is ample.
\end{coro}
In particular, Corollary~\ref{co:submfd} tells us that every nonsingular subvariety of a smooth compact quotient of the unit ball in $\mathbb{C}^n$ has ample canonical bundle.
Corollary~\ref{co:submfd} partially answers a question asked by Lang~\cite[p. 162]{Lang}.

\subsection*{Notation}
For any smooth volume form $V$ on $X^n$ with $V = \gamma_{\alpha} \prod_i dz_{\alpha}^i \wedge d\bar{z}_{\alpha}^i $ in a local chart $(U_{\alpha}; z_{\alpha}^1,\ldots, z_{\alpha}^n)$, we denote
\[
    dd^c \log V = dd^c \log \gamma_{\alpha} = \frac{\sqrt{-1}}{2\pi} \p \bar{\p} \log \gamma_{\alpha}
\]
where $d^c = \sqrt{-1} (\bar{\p} - \p)/ (4\pi)$. The $(1,1)$ form $dd^c\log V$ is globally defined and represents the first Chern class $c_1(K_X)$.
For a metric $\omega$ on $X$, its Ricci form is
\[
   \Ric(\omega) = - dd^c \log \omega^n.
\]
\section{Reduction to bigness}


The proof of Theorem~\ref{th:WY} can be reduced to show that $K_X$ is big, i.e., $X$ is of general type. More precisely, we have the following result.
\begin{lemma}\label{le:bigint}
Let $X$ be a  projective manifold $X$ of complex dimension $n$, and let $\omega$ be a K\"ahler metric on $X$ whose holomorphic sectional curvature $H$ is negative everywhere. Then $K_X$ is nef. If in addition
\begin{equation} \label{eq:topKx}
    \int_X c_1(K_X)^n > 0,
\end{equation}
 then $K_X$ is ample.
\end{lemma}
\begin{proof}
Since $H < 0$, there is no non-constant holomorphic map $\mathbb{C} \to X$, by the general Schwarz lemma (see \cite[Corollary 1]{Roy80}. In particular, this implies that $X$ does not contain any rational curves (see \cite[p. 37]{Laz2} and \cite[Theorem 2.1]{Dem}). 
Then $K_X$ is nef, by the Cone Theorem (see \cite[p. 22]{KoMo} for example). Condition \eqref{eq:topKx} means that the top self-intersection of the canonical divisor is positive. This together with the nefness of $K_X$ imply that $K_X$ is big (see for example \cite[p. 144, Theorem 2.2.16]{Laz1}). That is, $X$ is of general type. Since $X$ contains no rational curve, we conclude that $K_X$ is ample (see \cite[p. 219]{Debar} for example).
\end{proof}

\section{Proof of ampleness}
Theorem~\ref{th:WY} follows immediately from Lemma~\ref{le:bigint} and the following result. All results in this section hold for compact K\"ahler manifolds.
\begin{lemma} \label{le:nefbig}
Let $(X,\omega)$ be an $n$-dimensional compact K\"ahler manifold with negative holomorphic sectional curvature. If $K_X$ is nef,  then
\begin{equation} \label{eq:c1Lk}
   \int_X c_1(K_X)^{k} \wedge \omega^{n-k} > 0 \quad \textup{on $X$}
\end{equation}
for all $1 \le k \le n$. In particular, $K_X$ is big.
\end{lemma}

Lemma~\ref{le:nefbig} can be strengthened below to directly obtain the ampleness of $K_X$. This gives an alternative proof of Theorem~\ref{th:WY}, using only the first statement of Lemma~\ref{le:bigint}.
\begin{theorem} \label{th:WY2}
Let $(X,\omega)$ be an $n$-dimensional compact K\"ahler manifold with negative holomorphic sectional curvature. If $K_X$ is nef, then there exists a smooth function $u$ on $X$ such that $dd^c \log \omega^n + dd^c u$ is the K\"ahler-Einstein metric on $X$ of negative Ricci curvature. As a consequence, $K_X$ is ample.
\end{theorem}

To show Lemma~\ref{le:nefbig} and Theorem~\ref{th:WY2}, we first consider the case that $K_X$ is nef, without any assumption on the holomorphic sectional curvature.
\begin{prop}\label{pr:nefL2}
Let $(X, \omega)$ be an $n$-dimensional compact K\"ahler manifold. If $K_X$ is nef, then the following properties hold:
\begin{enumeratei}
\item \label{it:nefL1} For every $\varepsilon>0$ there exists a smooth function $u_{\varepsilon}$ on $X$ such that 
\[
   \omega_{\varepsilon} \equiv \varepsilon \omega + dd^c \log \omega^n + dd^c u_{\varepsilon} > 0 \quad \textup{on $X$},
\]
and \, $\omega_{\varepsilon}^n = e^{u_{\varepsilon}} \omega^n$ on $X$. Furthermore, 
\begin{align}
   \Ric (\omega_{\varepsilon}) & = - \omega_{\varepsilon} + \varepsilon \omega \ge - \omega_{\varepsilon}, \label{eq:lowRic}
\end{align}
and
\begin{equation} \label{eq:uppu}
   \sup_X u_{\varepsilon} \le C,
\end{equation}
where the constant $C>0$ depends only on $\omega$ and $n$.
\item \label{it:nefL2} For each $k = 1, \ldots, n$, 
\[
   \int_X c_1(K_X)^k \wedge  \omega^{n-k} \ge C^{k/n-1} \int_X c_1(K_X)^n \ge 0, 
\]
where the constant $C$ is the same as that in \eqref{eq:uppu}.
\end{enumeratei}
\end{prop}
\begin{proof}
First to show \eqref{it:nefL1}.
 Since $K_X$ is nef, for each $\varepsilon>0$, there exists a smooth function $f_{\varepsilon}$ on $X$ such that
  \[
     \omega_{f_{\varepsilon}} \equiv  \varepsilon \omega + dd^c \log \omega^n +  dd^c f_{\varepsilon} > 0 \quad \textup{on $X$}.
  \]
  The lower bound of the Ricci curvature of $\omega_{f_{\varepsilon}}$ may depend on $\varepsilon$. Fix an $\varepsilon>0$. By 
  Theorem 4~\cite[p. 383]{yau78-calabi} there exists a $v_{\varepsilon} \in C^{\infty}(X)$ satisfying
the Monge-Amp\`ere equation
  \begin{equation} \label{eq:MAfeps}
  (\omega_{f_{\varepsilon}} + dd^c v_{\varepsilon})^n  = \omega^n e^{v_{\varepsilon} + f_{\varepsilon}} %
  \end{equation}
with $\omega_{f_{\varepsilon}} + dd^c v_{\varepsilon} > 0$ on $X$. 
  Let 
  \[
     u_{\varepsilon} = f_{\varepsilon} + v_{\varepsilon}.
  \]
Then $\omega_{\varepsilon} = \omega_{f_{\varepsilon}} + dd^c v_{\varepsilon}$ satisfies $\omega_{\varepsilon}^n = e^{u_{\varepsilon}} \omega^n$, and furthermore \eqref{eq:lowRic}, for
  \begin{align*}
     \Ric (\omega_{f_{\varepsilon}} + dd^c v_{\varepsilon})
    & = - dd^c \log (\omega_{f_{\varepsilon}} + dd^c v_{\varepsilon})^n \\
     & = - dd^c \log \omega^n - dd^c f_{\varepsilon} -  dd^c v_{\varepsilon} \\
     & =  - (\omega_{f_{\varepsilon}} + dd^c v_{\varepsilon}) + \varepsilon \omega.
  \end{align*}
Apply the maximum principle to $u_{\varepsilon} = v_{\varepsilon} + f_{\varepsilon}$ in \eqref{eq:MAfeps} yields
\[
   \exp\Big( \sup_X u_{\varepsilon} \Big) \le C \equiv \sup_X \frac{(\varepsilon_0\omega +dd^c \log \omega^n )^n}{\omega^n},
\]
for all $\varepsilon < \varepsilon_0$. This implies \eqref{eq:uppu}.

For \eqref{it:nefL2}, the case $k = n$ follows immediately from the nefness of $K_X$. It suffices to consider $1 \le k \le n-1$. 
Denote
\[
   \sigma_k =\binom{n}{k} \frac{ \omega_{\varepsilon}^k \wedge \omega^{n-k}}{\omega^n}, \quad 0 \le k \le n.
\]
By Newton-MacLaurin's inequality
\[
   \Big[\frac{\sigma_n}{\sigma_0}\Big]^{1/n} \ge \bigg[\frac{\sigma_n}{\sigma_k / \binom{n}{k}} \bigg]^{1/(n-k)}
\]
and \eqref{eq:uppu} we obtain
\[
    \frac{\omega_{\varepsilon}^k \wedge \omega^{n-k}}{\omega_{\varepsilon}^n} = \frac{\sigma_k/ \binom{n}{k}}{\sigma_n} \ge \sigma_n^{k/n - 1} \ge C^{k/n-1}.
\]
Integrating against $\omega_{\varepsilon}^n$ over $X$ yields
\[
   \int_X \omega_{\varepsilon}^{k} \wedge \omega^{n-k} \ge C^{k/n-1} \int_X \omega_{\varepsilon}^n.
\]
This implies that
\[
   \int_X c_1(K_X)^{k} \wedge \omega^{n-k} \ge C^{k/n-1} \int_X c_1(K_X)^n + O(\varepsilon).
\]
Letting $\varepsilon \to 0^+$ yields the result.
\end{proof}
 
Next, we need the following inequality, which is implicitly contained in \cite{WWY} and \cite[pp. 370--372]{WYZ}.
\begin{prop} \label{pr:WYZ}
  Let $M$ be a K\"ahler manifold of complex dimension $n$, and let $\omega$ and $\omega'$ be two K\"ahler metrics on $X$. Suppose that the holomorphic sectional curvature of $\omega$ satisfies
  \[
     H(P; \eta) \equiv R_P(\eta, \bar{\eta}, \eta, \bar{\eta}) \le -\kappa(P) \|\eta\|^4_P \quad \textup{for all $\eta \in T_P X$ and all $P \in X$},
  \]
 and that the Ricci curvature of $\omega'$ satisfies
  \[
     \Ric (\omega') \ge - \lambda \omega' + \mu \omega, 
  \]
 where $\kappa$, $\mu$, and $\lambda$ are continuous functions with $\kappa \ge 0$ and $\mu \ge 0$ on $X$.  Let $S$ be the trace of $\omega$ with respect to $\omega'$; that is,
 \[
     S = \frac{n (\omega')^{n-1} \wedge \omega}{(\omega')^n}.
 \]
  Then 
  \begin{equation*}
     \Delta' \log S \geq  \Big[\frac{(n+1)\kappa}{2n} + \frac{\mu}{n} \Big] S  - \lambda,
  \end{equation*}
  where $\Delta'$ is the Laplacian of $\omega'$.
\end{prop}
\begin{proof}
We denote by $g_{i\bar{j}}$, $R_{i\bar{j}}$, and $R_{i\bar{j}k\bar{l}}$ the components of metric tensor, Ricci curvature tensor, and the curvature tensor of $\omega$, respectively; and similarly by $g'_{i\bar{j}}$, $R'_{i\bar{j}}$, and $R'_{i\bar{j}k\bar{l}}$ the corresponding tensors of $\omega'$. Choose a normal coordinate system $(z^1,\ldots, z^n)$ near a point $P$ of $M$ such that
\[
   g_{i\bar{j}} = \delta_{ij}, \quad \frac{\p}{\p z^k} g_{i\bar{j}} = 0, \quad g'_{i\bar{j}} = \delta_{ij} g'_{i\bar{i}}
\]
at $P$. By \cite[p. 623, (2.3)]{WWY} we have
\begin{equation} \label{eq:LapS}
   \Delta' S =   \sum_{i} \frac{R'_{i\bar{i}}}{(g'_{i\bar{i}})^2}  + \sum_{i,j,k} \frac{|\partial g'_{i\bar{j}}/\partial z^k|^2}{g'_{i\bar{i}} (g'_{j\bar{j}})^2 g'_{k\bar{k}}} - \sum_{i,k} \frac{R_{i\bar{i}k\bar{k}}}{g'_{i\bar{i}} g'_{k\bar{k}}} \qquad \textup{at $P$}. 
\end{equation}
As in \cite[p. 372]{WYZ} we apply the Cauchy-Schwarz inequality to obtain
\begin{equation}\label{eq:3rdt}
\sum_{i,j,k} \frac{|\partial g'_{i\bar{j}}/\partial z^k|^2}{g'_{i\bar{i}} (g'_{j\bar{j}})^2 g'_{k\bar{k}}}
\ge \sum_{i,k} \frac{|\p g'_{i\bar{i}}/ \p z^k|^2}{(g'_{i\bar{i}})^3 g'_{k\bar{k}}}\ge \frac{|\nabla' S|^2}{S}.
\end{equation}
It follows from Royden's lemma~\cite[p. 624]{WWY} that
\begin{equation} \label{eq:Ht}
   - \sum_{i,k} \frac{R_{i\bar{i}k\bar{k}}}{g'_{i\bar{i}} g'_{k\bar{k}}}
   \ge \frac{(n+1)\kappa}{2n} S^2.
\end{equation}
By the assumption on Ricci curvature we have
\begin{equation} \label{eq:Rict}
  \sum_i \frac{R'_{i\bar{i}}}{(g'_{i\bar{i}})^2} \ge - \lambda S + \mu \sum_i \frac{1}{(g'_{i\bar{i}})^2} \ge - \lambda S + \frac{\mu}{n} S^2.
\end{equation}
Plugging \eqref{eq:3rdt}, \eqref{eq:Ht}, and \eqref{eq:Rict} into \eqref{eq:LapS} yields the desired inequality.
\end{proof}

We are now ready to prove Lemma~\ref{le:nefbig}  and Theorem~\ref{th:WY2}.

\noindent \emph{Proof of Lemma~\ref{le:nefbig}.}
Since $K_X$ is nef,  by Proposition~\ref{pr:nefL2} \eqref{it:nefL1} for each $\varepsilon>0$   there exists a K\"ahler metric $\omega_{\varepsilon} = \varepsilon \omega + dd^c \log \omega^n + dd^c u_{\varepsilon}$ on $X$ such that
\[
   \Ric(\omega_{\varepsilon}) \ge -\omega_{\varepsilon} \quad \textup{and} \quad \omega_{\varepsilon} \le C \omega^n \quad \textup{on $X$},
\]
where $C>0$ is a constant independent of $\varepsilon$. Let
\[
   S_{\varepsilon} = \frac{n\omega_{\varepsilon}^{n-1} \wedge \omega}{\omega_{\varepsilon}^n}. 
\]
By Proposition~\ref{pr:WYZ} we have
\begin{equation} \label{eq:C2ineq}
   \Delta' \log S_{\varepsilon} \ge \frac{(n+1)\kappa}{2 n} S_{\varepsilon} - 1 \quad \textup{on $X$}.
\end{equation}
Here
\[
   \kappa(P) 
   = - \sup_{\eta \in T_p X \setminus\{0\} } \frac{R(\eta, \bar{\eta}, \eta, \bar{\eta})}{\|\eta\|^4_{P, \omega}} > 0 \quad \textup{for all $P \in X$}.
\]
Since $X$ is compact, there exists a constant $\kappa_0 > 0$ such that $\kappa \ge \kappa_0 > 0$. Applying the maximum principle to $S_{\varepsilon}$ in \eqref{eq:C2ineq} yields
\[
    S_{\varepsilon} \le \frac{2n}{(n+1)\kappa_0} \quad \textup{on $X$}.
\] 
On the other hand, by Newton-Macluarin's inequality
\[
   S_{\varepsilon} = \frac{\sigma_{n-1}}{\sigma_{n}} \ge n \sigma_n^{-1/n}
\]
where $\sigma_k \equiv \binom{n}{k} \omega_{\varepsilon}^k \wedge \omega^{n-k} / \omega^n$ for each $1 \le k \le n$.
It follows that
\[
   \int_X \omega_{\varepsilon}^n = \int_X \sigma_n \omega^n \ge \frac{(n+1)^n}{2^n} \kappa_0^n \int_X \omega^n > 0.
\]
Note that
\[
   \int_X \omega_{\varepsilon}^n = \int_X c_1(K_X)^n + \varepsilon n \int_X c_1(K_X)^{n-1} \wedge \omega + O(\varepsilon^2).
\]
Letting $\varepsilon \to 0^+$ yields
\[
   \int_X c_1(K_X)^n \ge \frac{(n+1)^n}{2^n} \kappa_0^n \int_X \omega^n > 0. 
\]
This implies that $K_X$ is big. This and Proposition~\ref{pr:nefL2} \eqref{it:nefL2} imply \eqref{eq:c1Lk}.
\qed

\noindent  \emph{Proof of Theorem~\ref{th:WY2}}. As in the proof of Lemma~\ref{le:nefbig}, we have for each $\varepsilon>0$ a K\"ahler metric $\omega_{\varepsilon} = \varepsilon \omega + dd^c \log \omega^n + dd^c u_{\varepsilon}$ satisfying $\omega^n_{\varepsilon} =  e^{u_{\varepsilon}} \omega^n$ and
\[
   \textup{Ric}(\omega_{\varepsilon}) = - \omega_{\varepsilon} + \varepsilon \omega, \quad \max_X u_{\varepsilon} \le C, \quad C^{-1} \le S_{\varepsilon} \le C.
\]
We denote by $C>0$ a generic constant depending only on $\omega$ and $n$. Since
\[
  S_{\varepsilon}^{n-1} \ge \sigma_1/ \sigma_n,
\]
We obtain $\sigma_1 \le C$, and hence $C^{-1} \omega \le \omega_{\varepsilon} \le C \omega$ and a uniform estimate for $u_{\varepsilon}$ up to the second order. Then, a standard process shows that $\|u_{\varepsilon}\|_{C^{k,\alpha}(X)} \le C$ for any nonnegative integer $k$ and $0 < \alpha <1$ (see \cite[pp. 360 and 363]{yau78-calabi} for example). Thus, there is a sequence $\{u_{\varepsilon_l}\}$ converges in the $C^{k,\alpha}(X)$-norm to a solution $u$ of the equation
\[
    (dd^c \log \omega^n + dd^c u)^n = e^u \omega^n
\]
with $\omega_u \equiv dd^c \log \omega^n + dd^c u > 0$ on $X$.
This implies that $\omega_u$ is the unique K\"ahler-Einstein metric with $\textup{Ric}(\omega_u) = - \omega_u$. \qed

\end{document}